\documentclass[11pt]{article}
\usepackage{amsfonts}
\usepackage{amssymb}

\usepackage{amssymb,amsmath,amsthm,epsfig}

\theoremstyle{plain}
\newtheorem{thm}{Theorem}[section]
\newtheorem{lem}[thm]{Lemma}

\newtheorem*{mthm}{Main Theorem}

\theoremstyle{definition}
\newtheorem{defn}[thm]{Definition}

\newtheorem{rem}{Remark}[section]
\newtheorem{exm}{Example}[section]

\numberwithin{equation}{section}

\usepackage{graphics}

\begin{document}

\title{Devaney's chaos revisited\footnote{Supported in
part by the National Natural Science Foundation of P.\ R.\ China
(11071263)}}

\author{Xiaoyi Wang and Yu Huang\footnote{Corresponding author: Y. Huang, E-mail:
stshyu@mail.sysu.edu.cn}}

\date{}

\maketitle

\begin{center}
Department of Mathematics, Sun Yat-sen University, Guangzhou 510275,
P.\ R.\ China
\end{center}

\baselineskip 16.2pt

\begin{abstract}
Let $X$ be a metric space, and let $f\colon X\rightarrow X$ be a
continuous transformation. In this note, a concept
\textit{indecomposability} of $f$ is introduced. We show that
transitivity implies indecomposability and that Devaney's chaos is
equivalent to indecomposability together with density of periodic
points. Moreover, we point out that the indecomposability and the
periodic-points density are independent of each other even for
interval maps (i.e., neither implies the other).

\vskip 0.1cm \noindent {\bf AMS} classification: 54H20; 37D45


\end{abstract}

\section{Introduction}\label{sec1}
Throughout this note, we let $f\colon X\rightarrow X$ be a continuous transformation of a metric space $X$.
Devaney~\cite{de} called it to be \textit{chaotic} if it satisfies the following three conditions:
\begin{itemize}
\item[(i)] $f$ is transitive.
\item[(ii)] the  periodic points of $f$ are dense in $X$.
\item[(iii)] $f$ has sensitive dependence on initial conditions.
\end{itemize}
As the concepts entropy and Li-Yorke's chaos, Devaney's chaos is an
important tool to discover the complexity of the dynamical system
$(X,f)$. These three concepts have intrinsic relations each
other~\cite{AC,BG,HY}.

It is well known that in Devaney's chaos, conditions (i), (ii) and (iii) are not
independent of each other; see, for examples, Banks et al.~\cite{ba},
Assaf IV and Gadbois~\cite{as},  Vellekoop and Berglund~\cite{ve},
Crannell~\cite{ca}, and Touhey~\cite{to}.

In this note, we further study Devaney's chaos by introducing the
concept---\textit{indecomposability}. Two $f$-invariant closed
subsets $A, B$ are \textit{independent} if they have no common
interior points; that is,
$\mathrm{Int}(A)\cap\mathrm{Int}(B)=\varnothing$. Now $(X,f)$ is
called \textit{indecomposable} if any two $f$-invariant closed
subsets having nonempty interiors are not independent.

Since transitivity is equivalent to the fact that the only
$f$-invariant closed subset of $X$ having nonempty interior is $X$
itself, transitivity implies indecomposability. The converse however
is not true even in the case of $X=[0,1]$, as shown by
Example~\ref{ex3.1} below.

Our main result proved in this note can be stated as follows:

\begin{mthm}
For the topological dynamical system $(X,f)$, the following two
statements are equivalent to each other:
\begin{enumerate}
\item[$(1)$] $f$ is Devaney chaotic.

\item[$(2)$] $f$ is indecomposable and the periodic points are dense in $X$.
\end{enumerate}
\end{mthm}

We will prove this in Section~\ref{sec2}. An interesting point is that
the indecomposability and the periodic-points density are independent of
each other even for interval maps (i.e., neither implies the other), as shown by Example~\ref{ex3.2} below.

%%%%%%%%%%%%%%%%%%%%%%%%%%%%%%%%%%%%%%%%%%%%%%%%%%%%%%%%%%%%%%%%%%%%%%%%%%
%%%%%%%%%%%%%%%%%%%%%%%%%%%%%%%%%%%%%%%%%%%%%%%%%%%%%%%%%%%%%%%%%%%%%%%%%%
\section{Several equivalent definitions for Devaney's chaos}\label{sec2}

Let $(X,f)$ be a topological dynamical system as in Section~\ref{sec1}.
To avoid the trivial case, we assume that $X$ has at least infinitely many elements.

We denote respectively the recurrent points set, orbit and $\omega$-limit points set of $f$ by
\begin{gather*}
R(f)=\left\{x\in X\,|\,\exists n_k\uparrow\infty\textrm{ s.t. }f^{n_k}(x)\to x\right\},\\
\mathrm{Orb}_f(x)=\left\{f^n(x)|n\geq 0\right\},\\
\omega(x)=\bigcap\limits_{n\geq 0}\overline{\{f^k(x)|k\geq n\}}.
\end{gather*}
Given a subset $A\subset X$, we denote the interior of $A$ by
$\mathrm{Int}(A)$ and the closure of $A$ by $\overline{A}$ in $X$.
$A$ is invariant if $f(A)\subset A$. A subset $S$ of $X$ is residual
if it contains a dense $G_\delta$ set. A Baire space is a
topological space such that every nonempty open subset is of second
category.

Recall that $f$ is transitive if for any two nonempty open subsets $U$ and $V$ in $X$ there exists $n\in \mathbb{Z}_{+}$
such that $f^n(U)\cap V\neq\varnothing$. A point $x\in X$ is called a transitive point if
$\overline{\mathrm{Orb}_f(x)}=X$. By $\mathrm{Tr}_f$ we mean the set of all transitive points of
$f$. It is well known that $\mathrm{Tr}_f$ is a dense $G_\delta$ set if $X$
is a Baire separable metric space.

\begin{defn}
Let $f\colon X\rightarrow X$ be a continuous transformation on the metric space $X$. $f$ is said to be
\begin{itemize}
 \item[(i)] \textit{strongly indecomposable} if for any sequence of
$f$-invariant closed subsets $\{A_n\}_{n=1}^\infty$ of $X$ with
$\mathrm{Int}(A_n)\not=\varnothing$,
$\mathrm{Int}(\bigcap_{n=1}^\infty A_n)\not=\varnothing$;

 \item[(ii)] \textit{indecomposable} if for any two $f$-invariant closed subsets
$A,B\subset X$ with $\mathrm{Int}(A)\neq\varnothing$ and
$\mathrm{Int}(B)\neq\varnothing$, $\mathrm{Int}(A\cap
B)\not=\varnothing$;

\item[(iii)] \textit{weakly indecomposable} if there exists a residual
subset $S\subset X$ such that for any two points $x,y\in S$,
$\omega(x)=\omega(y)\neq\varnothing$.
\end{itemize}
\end{defn}

It is easily seen that the following
implication relations hold:
\[
 transitivity\Rightarrow strongly\ \ indecomposability\Rightarrow
 indecomposability.
\]
We will show that indecomposability implies weakly indecomposability
provided that $X$ is a compact space (see Theorem~\ref{thm3.1}
below). And we will give examples in Section~\ref{sec3} to show all
the converses are not true.

\begin{lem}\label{lem2.2}
Let $f\colon X\rightarrow X$ be a continuous transformation on the metric space $X$ such that
$X=\overline{R(f)}$.
Then the following conditions are equivalent:
\begin{enumerate}
\item[$(1)$] $f$ is  transitive.
\item[$(2)$] $f$ is strongly indecomposable.
\item[$(3)$] $f$ is indecomposable.
\end{enumerate}
\end{lem}

\begin{proof}
Obviously,
$\mathrm{(1)}\Rightarrow\mathrm{(2)}\Rightarrow\mathrm{(3)}$. Now we
prove $\mathrm{(3)}\Rightarrow\mathrm{(1)}$. It suffices to show
that for any closed invariant subset $A$ of $X$ with nonempty
interior, we have $A=X$, under the condition $X=\overline{R(f)}$. In
fact, as $f$ is indecomposable, for any nonempty open set $V\subset
X$, we have $\overline{\bigcup_{n\geq 0}f^n(V)}\cap A$ has nonempty
interior. Then there exist nonempty open set $V_1\subset V$ and
$n\in \mathbb{Z}_+$ such that $f^n(V_1)\subset\mathrm{Int}(A)$.
Since the recurrent points of $f$ are dense in $V_1$ and $A$ is an
invariant closed set, we have $V_1\subset A$. By the arbitrariness
of $V$, we get $X\subset A$. Thus, $f$ is transitive.

This proves Lemma~\ref{lem2.2}.
\end{proof}

\begin{lem}\label{lem2.3}
Let $f\colon X\rightarrow X$ be a continuous transformation on a Baire separable metric space $X$ such that
$X=\overline{R(f)}$.
Then the following conditions are equivalent:
\begin{enumerate}
\item[$(1)$] $f$ is  transitive.
\item[$(2)$] $f$ is strongly indecomposable.
\item[$(3)$] $f$ is indecomposable.
\item[$(4)$] $f$ is weakly indecomposable.
\end{enumerate}
\end{lem}

\begin{proof}
According to Lemma~\ref{lem2.2}, we need only prove
$\mathrm{(1)}\Leftrightarrow\mathrm{(4)}$. Since $X$ is a Baire
separable metric space, $\mathrm{Tr}_f$ is a dense $G_\delta$ set.
For any two points $x, y\in\mathrm{Tr}_f$, $\omega(x)=\omega(y)=X$.
Thus $\mathrm{(1)}\Rightarrow\mathrm{(4)}$ holds. Conversely, assume
$f$ is weakly indecomposable. Let $S$ be the residual set such that
for any two points $x,y\in S,\ \omega(x)=\omega(y)$. since $R(f)$ is
a dense $G_\delta$ set, $S\cap R(f)$ is residual. For any points
$x,y\in S\cap R(f)$, $y$ is recurrent and $y \in
\omega(y)=\omega(x)$. As $\omega$-limit set of $x$ is closed, we
have $\omega(x)=X$. Thus $f$ is transitive.

This proves Lemma~\ref{lem2.3}.
\end{proof}

From the statements of Lemmas~\ref{lem2.2} and \ref{lem2.3}, we easily get the following two results.

\begin{thm}\label{thm2.4}
Let $f\colon X\rightarrow X$ be a continuous transformation on the metric space $X$. Then the following conditions are equivalent:
\begin{enumerate}
\item $f$ is Devaney chaotic.
\item $f$ is transitive and has a dense set of periodic points.
\item $f$ is strongly indecomposable and has a dense set of periodic points.
\item $f$ is indecomposable and has a dense set of periodic points.
\end{enumerate}
\end{thm}

\begin{thm}\label{thm2.5}
Let $f\colon X\rightarrow X$ be a continuous transformation on a Baire separable metric space $X$. Then the following conditions are equivalent:
\begin{enumerate}
\item $f$ is Devaney chaotic.

\item $f$ is weakly indecomposable and has a dense set of periodic points.
\end{enumerate}
\end{thm}

Thus Theorem~\ref{thm2.4} implies our Main Theorem stated in Section~\ref{sec1}.

%%%%%%%%%%%%%%%%%%%%%%%%%%%%%%%%%%%%%%%%%%%%%%%%%%%%%%%%%%%%%%%%%%%%%%%%%%%%%%%%%%%%%%%%%%%%%
%%%%%%%%%%%%%%%%%%%%%%%%%%%%%%%%%%%%%%%%%%%%%%%%%%%%%%%%%%%%%%%%%%%%%%%%%%%%%%%%%%%%%%%%%%%%%
\section{Indecomposability and periodic-points density}\label{sec3}
Let $f\colon X\rightarrow X$ be a continuous transformation on the metric space $X$.
From now on, we let $U^*=\overline{\bigcup_{n\geq 0}f^{n}(U)}$ for any $U\subset X$, which is an invariant closed set of $f$.
Firstly we show that indecomposability implies weakly
indecomposability provided that $X$ is compact.

\begin{thm}\label{thm3.1}
Suppose $X$ is compact. If $f$ is indecomposable, then $f$ is weakly indecomposable.
\end{thm}

\begin{proof}
Let $\mathbb{B}=\{U_i\}_{i=1}^\infty$ be a topology basis of $X$. As
$f$ is indecomposable, for any $k\in \mathbb{Z}_+$, $\bigcap_{i=0}^k
U_{i}^*\neq\varnothing$. It follows from the compactness of $X$ that
$E=\bigcap_{i=0}^\infty U_i^*$ is nonempty, closed and invariant.
Let $\mathbb{B}_E=\{U\in\mathbb{B}|U\cap E\neq\varnothing\}$. Then
for any $U\in \mathbb{B}_E$, $\bigcup_{n=0}^\infty f^{-n}(U)$ is
open and dense in $X$. Thus $V_1=\bigcap_{U\in
\mathbb{B}_E}\bigcup_{n=0}^\infty f^{-n}(U)$ is a dense $G_\delta$
set of $X$. Taking any $x\in V_1$, we have
$\overline{\mathrm{Orb}_f(x)}\supset E$.

For any positive integer $k$, let $L_k=\mathrm{Int}(\bigcap_{i=1}^k
U_i^*)$ and $\Delta_k=\bigcup_{i=0}^\infty f^{-i}(L_k)$. As $f$ is
indecomposable, we have $L_k\neq\varnothing$. Thus $\Delta_k$ is
open and dense in $X$, and $\omega(x)\subset \bigcap_{i=1}^k U_i^* $
for any $x\in \Delta_k$. By Baire's theorem, it follows that
$V_2=\bigcap_{k=1}^\infty\Delta_k$ is a dense $G_\delta$ set.
Therefore for any point $x\in V_2$, $\omega(x)\subset
\bigcap_{i=1}^\infty U_i^*=E. $

We know that $V_1\cap V_2$ is a dense $G_\delta$ set from its
construction. We also have $\omega(x)\subset E\subset
\overline{\mathrm{Orb}_f(x)}$ for any $x\in V_1\cap V_2 $. There are
the following two cases.

Case 1, $\mathrm{Int}(E)\neq\varnothing$. Take
$x\in\mathrm{Int}(E)\cap V_1\cap V_2$ and $y\in V_1\cap V_2$. If $x$
is an isolate point, then $f^n(y)=x$ for some $n$ and
$\omega(x)=\omega(y)$. Otherwise $x$ must be a recurrent point of
$f$, then $E\supset\omega(y)\supset\omega(x)=E$, which implies
$\omega(x)=\omega(y)$. Thus $f$ is weakly indecomposable.

Case 2, $\mathrm{Int}(E)=\varnothing$. Let $E_0=\{e\in E|e$ is
isolate in $E \}$. That is, $E_0$ is the set of all points $e$ in
$E$ with $B_\varepsilon(e)\cap E=e$ for some $\varepsilon>0$, where
$B_\varepsilon(e)$ stands for open ball of center $e$ and radius
$\varepsilon$. $E_0$ must be countable. We claim that $f^{-n}(e)$
has empty interior for each $n>0$ and $ e\in E$. If not, let $n$ be
the smallest positive integer such that $f^{-n}(e)$ has nonempty
interior. When $f^{-n}(e)$ is a singleton, $f^{-n}(e)\in U_i^*$ for
each $i$ since $\overline{\bigcup_{i\geq -n}f^i(e)}$ is an invariant
closed subset with its interior being $f^{-n}(e)$  and $f$ is
indecomposable. Therefore $f^{-n}(e)\in E$ which contradicts
$\mathrm{Int}(E)=\varnothing$. When $f^{-n}(e)$ is not a singleton,
taking two disjoint nonempty open subset $\alpha,\beta\subset
f^{-n}(e)$, we have $\mathrm{Int}(\alpha^*\cap\beta^*)=\varnothing$
which contradicts to the indecomposability of $f$.

Since $f^{-n}(e)$ is a closed set with empty interior for $n>0$ and
$ e\in E$,  we have that $V_3=\bigcup_{e\in E_0}\bigcup_{n=0}^\infty
f^{-n}(e)$  is of first category and so $V_1\cap V_2-V_3$ is a
residual set. Take a point  $x\in V_1\cap V_2-V_3$, if $Orb_f(x)\cap
E=\varnothing$, then $E\subset \overline{Orb_f(x)}-Orb_f(x)\subset
\omega(x)\subset E$. Otherwise  there exists the smallest
nonnegative integer $n$ such that $f^n(x)\in E$. Then we have
\[
   \omega(x)\subset E\subset \overline{Orb_f(x)}-\{x,\ f(x),\cdots,\
   f^{n-1}(x)\}=
   \overline{Orb_{f^n}(x)}.
\]
And $f^n(x)$ must be a cluster point of $E$ because $x\notin V_3$.
Thus $f^n(x)$ is a recurrent point and
\[
 \omega(x)\subset E\subset \overline{Orb_{f^n}(x)}=
 \omega(f^n(x))=\omega(x).
\]
Therefore, in both cases we always have $\omega(x)= E$. This means
that $f$ is weakly indecomposable.

This completes the proof of Theorem~\ref{thm3.1}.
\end{proof}

Secondly, we show that strongly indecomposability is nearly
transitivity.

\begin{thm}\label{thm3.2}
If $X$ is a Baire separable metric space with no isolate point, $f:
X\rightarrow X$ is strongly indecomposable,  then there exist an
invariant closed set $E$ with nonempty interior such that  $f|_E$ is
transitive.
\end{thm}

\begin{proof}
Let $\mathbb{B}=\{U_i\}_{i=1}^\infty$ be a topology basis of $X$ and
$E=\bigcap_{i=1}^\infty U_i^*$. Then $\mathrm{Int}(E)\neq\varnothing$ by the
strongly indecomposability of $f$. Let
$\mathbb{B}_E=\{U\in\mathbb{B}|U\cap E\neq\varnothing\}$. Then for any
$U\in\mathbb{B}_E$, $\bigcup_{n=0}^\infty f^{-n}(U)$ is open and
dense in $X$, which implies that $\bigcap_{U\in
\mathbb{B}_E}\bigcup_{n=0}^\infty f^{-n}(U)$ is a dense $G_\delta$
set of $X$. Every point $x$ in $\mathrm{Int}(E)\cap\bigcap_{U\in
\mathbb{B}_E}\bigcup_{n=0}^\infty f^{-n}(U)$ is a transitive point
of $f|_E$. Therefore, $f|_E$ is transitive since $X$ has no isolate
point.
\end{proof}

Finally, we consider one-dimensional system $(I, f)$, where $I$ is
an interval and $f$ is a continuous map of $I$. We need a lemma.

\begin{lem}[ \cite{ve}]\label{lem3.3}
Suppose  $f:I\rightarrow I$ is an interval map. If $J\subset I$ is a
subinterval containing no periodic point and $z,f^m(z),f^n(z) \in J$
with $0<m<n$,  then either $z<f^m(z)<f^n(z)$ or $z>f^m(z)>f^n(z)$.
\end{lem}

\begin{thm}\label{thm3.4}
An interval map $f:I\rightarrow I$ is strongly indecomposable, then
there exist a positive integer $n$ and disjoint closed non
degenerate subintervals $J_0,J_1,\dots,J_{n-1}, J_n=J_0$ such that
$f(J_i)=J_{i+1},~i=0,\dots,n-1$ and $f$ is Devaney chaotic
on $\bigcup\limits_{i=0}\limits^{n-1} J_i$. Furthermore, $f^n$ is
Devaney chaotic on $J_i,~i=0,\dots,n-1$.
\end{thm}

\begin{proof}
Assume that $f$ is strongly indecomposable. By Theorem \ref{thm3.2},
there exists a closed subset $E$ which contains a non degenerate
interval $J$ such that $f|_E$ is transitive. Let $(a,b)$ be any non
degenerate subinterval of $J$. Suppose $(a,b)$ contains  no periodic
point. By the transitivity of $f|_E$, there exist a transitive point
$x\in (a,b)$ and $0<p<q$ such that $a<f^q(x)<x<f^p(x)<b$, which
contradicts with Lemma \ref{lem3.3}. Thus the periodic points are
dense in $J$. Since $f|_E$ is transitive, the periodic points are
dense in $E$, $f|_E$ is Devaney chaotic.

Let $J_0$ be the longest subinterval of $E$. It must be closed
because $E$ is closed. Since $f|_E$ is  transitive, there exists the
smallest positive integer $n$ such that $f^n(J_0)\cap J_0\neq
\varnothing$. We have $f^n(J_0)\subset J_0$ as $J_0$ is the longest
subinterval. The transitivity of $f|_E$ ensures $f^n(J_0)= J_0$. Let
$J_i=f^{i}(J_0),~i=0,\dots,n-1$. We claim that $J_i,\ i=0,\dots,n-1$
are disjoint. If not, there exist integers $0\leq l<m\leq n-1$ such
that $J_l\cap J_m\neq \varnothing$, which follows $J_0\cap
J_{m-l}\supset f^{n-l}(J_l\cap J_m)\neq\varnothing$. A
contradiction.

Since $J_0$ is closed and $f^{n-i}(J_i)=J_0$, $J_i$ is closed for
$i=0,\dots,n-1$. It follows that
$\bigcup\limits_{i=0}\limits^{n-1}J_i$ is invariant and closed. So
$E=\bigcup\limits_{i=0}\limits^{n-1}J_i$.  It's not difficult to
check that $f^n|_{J_i}$ is
 transitive and  chaos in the sense of Devaney.
\end{proof}

At the end of this note, we give two examples to illustrate that
\[
  \textrm{indecomposability} \nRightarrow \textrm{strongly indecompossability}\nRightarrow \textrm{transitivity}.
\]

\begin{exm}\label{ex3.1}
Let $I=[0,1]$ and $f$ be defined as
\begin{equation*}
f(x)= \begin{cases}
-2x+1,& x\in [0,\frac{1}{6}],\\
2x+1/3,& x\in [\frac{1}{6},\frac{1}{3}],\\
-3x+2,& x\in [\frac{1}{3},\frac{2}{3}],\\
x-2/3,& x\in [\frac{2}{3},1].
\end{cases}
\end{equation*}
See Figure \ref{fig3}. Then $f$ is strongly indecomposable but not
transitive.
\end{exm}

\begin{figure}[h]
\centerline{\epsfig{figure=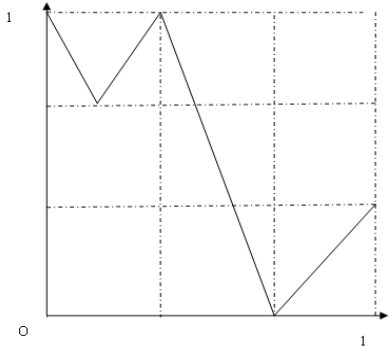,width=6cm,height=5cm}} \caption{The
profile of $f$ in Example~\ref{ex3.1}} \label{fig3}
\end{figure}

\begin{proof}
We show that $f$ is strongly indecomposable.
On interval $[0,\frac{1}{3}]$, $f^2$ can be expressed as
\begin{equation*}
f^2(x)=\begin{cases}
-2x+\frac{1}{3},& x\in [0,\frac{1}{6}],\\
2x-\frac{1}{3},& x\in [\frac{1}{6},\frac{1}{3}].
\end{cases}
\end{equation*}
It is clear that $f^2|_{[0,\frac{1}{3}]}$ is mixing. For any non
degenerate subinterval $J\subset [0,1]$, there exists an integer $
n\geq 0$ such that $f^n(J)\cap (0,\frac{1}{3})\neq\varnothing$.
Since $f^n(J)$ is non degenerate, $\overline{\bigcup_{n\geq
0}f^n(J)}\supset [0,\frac{1}{3}]$, $f$ is strongly indecomposable on
$[0,1]$. On the other hand, for $x\in (\frac{3}{9},\frac{4}{9})$,
$f^n(x)$ never comes back into $(\frac{1}{3},\frac{2}{3})$ any more
for $n>0$. Thus $f$ is not transitive.
\end{proof}

\begin{figure}[h]
\centerline{\epsfig{figure=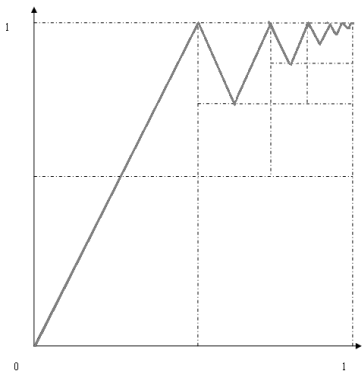,width=5cm,height=5cm}} \caption{The
profile of $f$ on [0,1] in Example~\ref{ex3.2}.} \label{fig1}
\end{figure}

\begin{exm}\label{ex3.2}
Let $I=[0,1]$ and $f: I\rightarrow I$ be defined as
\begin{gather*}
f(0)=0; f(1)=1;\\
f(1-\frac{1}{2^n})=1,~~n=1,2,\dotsc,\\
f(1-\frac{3}{2^{n+2}})=1-\frac{1}{2^{n+1}},~~n=1,2,\dotsc.
\end{gather*}
$f$ is linear between $1-\frac{1}{2^n}$ and $1-\frac{3}{2^{n+2}}$,
$n=1,2,\dotsc$. See Figure~\ref{fig1}. Then $f$ is indecomposable
but not strongly indecomposable. Furthermore, $f$ has only two
periodic points $0$ and $1$.
\end{exm}

\begin{proof}
To show that $f$ is indecomposable, let
 $A,B\subset X$ be two invariant closed subsets with non degenerate
intervals $I\subset A, J\subset B$, respectively. If $I$ covers at
least 2 critical points, then $f(I)\supset [1-\frac{|I|}{4},1]$.
Here $|I|$ denotes the length of the interval $I$. If $I$ covers
less than two critical points, we have $|f(I)|\geq |I|$ by the fact
that the absolute value of slope of $f$ is 2 everywhere except the
critical points. Thus there exists the smallest positive integer $n$
such that $f^n(I)$ covers at least two critical points. We have
$f^{n+1}(I)\supset [1-\frac{|I|}{4},1]$. Hence $A\supset
[1-\frac{|I|}{4},1] $. Similarly, we have $B\supset
[1-\frac{|J|}{4},1] $. Therefore, $A\cap B$ contains a non
degenerate interval and $f$ is indecomposable.

But  $f$ is not strongly indecomposable. In fact, let
$J_n=[1-\frac{1}{2^n},1],n=1,2,\dots$. Then $J_n$ is invariant with
nonempty interiors and $\bigcap\limits_n J_n=\{1\}$ which contains
no interior point.

It is easily  seen that the  periodic points of $f$ are
 $\{0,1\}$ and the point $\{1\}$ attracts all the points except the origin. $f$ is far
from chaos.
\end{proof}

\begin{rem}
For interval maps $f: I\rightarrow I$, strongly indecomposability
does not imply periodic-points density on $I$, but it ensures
Devaney's chaos on some subintervals of $I$.
\end{rem}

\begin{rem}
Even for interval maps, indecomposability does not imply
periodic-points density. Example~\ref{ex3.2} demonstrates that  an
indecomposable interval map can be far from chaos.
\end{rem}

\begin{rem}
Weakly indecomposability is the weakest concept among the three ones
on compact space. Such system has a topologically ``large" set of
points, each of which has the same $\omega$-limit set. An
indecomposable map can have very simple dynamics. For example, any
constant map on a metric space $X$ (a map which maps all points of
$X$ into a common fixed point) is indecomposable.
\end{rem}

\section*{Acknowledgment}
The authors would like to thank Professor Xiongping Dai for some
valuable discussion.

\end{document}